\documentclass{amsart}
\usepackage{amssymb}
\usepackage{ifthen}
\usepackage{graphicx}

\newtheorem{thm}{Theorem}
\newtheorem{cor}{Corollary}
\newtheorem{lem}{Lemma}

\newtheorem{rem}{Remark}

\newtheorem{conj}{Conjecture}
\theoremstyle{definition}
\newtheorem{example}{Example}
\newtheorem{prob}{Problem}

\newcommand{\A}{{\mathcal A}}

\newcommand{\U}{{\mathcal U}}
\newcommand{\es}{{\mathcal S}}

\newcommand{\IC}{{\mathbb C}}
\newcommand{\ID}{{\mathbb D}}

\newcommand{\D}{{\mathbb D}}

\def\be{\begin{equation}}
\def\ee{\end{equation}}

\newcommand{\bee}{\begin{enumerate}}
\newcommand{\eee}{\end{enumerate}}

\newcommand{\blem}{\begin{lem}}
\newcommand{\elem}{\end{lem}}
\newcommand{\bthm}{\begin{thm}}
\newcommand{\ethm}{\end{thm}}
\newcommand{\bcor}{\begin{cor}}
\newcommand{\ecor}{\end{cor}}
\newcommand{\beg}{\begin{example}}
\newcommand{\eeg}{\end{example}}
\newcommand{\begs}{\begin{examples}}
\newcommand{\eegs}{\end{examples}}
\newcommand{\bdefe}{\begin{defin}}
\newcommand{\edefe}{\end{defin}}
\newcommand{\bprob}{\begin{prob}}
\newcommand{\eprob}{\end{prob}}
\newcommand{\bei}{\begin{itemize}}
\newcommand{\eei}{\end{itemize}}

\newcommand{\bcon}{\begin{conj}}
\newcommand{\econ}{\end{conj}}
\newcommand{\bcons}{\begin{conjs}}
\newcommand{\econs}{\end{conjs}}
\newcommand{\bprop}{\begin{propo}}
\newcommand{\eprop}{\end{propo}}
\newcommand{\br}{\begin{rem}}
\newcommand{\er}{\end{rem}}
\newcommand{\brs}{\begin{rems}}
\newcommand{\ers}{\end{rems}}
\newcommand{\bo}{\begin{obser}}
\newcommand{\eo}{\end{obser}}
\newcommand{\bos}{\begin{obsers}}
\newcommand{\eos}{\end{obsers}}
\newcommand{\bpf}{\begin{pf}}
\newcommand{\epf}{\end{pf}}
\newcommand{\ba}{\begin{array}}
\newcommand{\ea}{\end{array}}
\newcommand{\beq}{\begin{eqnarray}}
\newcommand{\beqq}{\begin{eqnarray*}}
\newcommand{\eeq}{\end{eqnarray}}
\newcommand{\eeqq}{\end{eqnarray*}}

\begin{document}
\title[New upper bounds of the third Hankel determinant]{New upper bounds of the third Hankel determinant for some classes of univalent functions}

\author[M. Obradovi\'{c}]{Milutin Obradovi\'{c}}
\address{Department of Mathematics,
Faculty of Civil Engineering, University of Belgrade,
Bulevar Kralja Aleksandra 73, 11000, Belgrade, Serbia}
\email{obrad@grf.bg.ac.rs}

\author[N. Tuneski]{Nikola Tuneski}
\address{Department of Mathematics and Informatics, Faculty of Mechanical Engineering, Ss. Cyril and Methodius
University in Skopje, Karpo\v{s} II b.b., 1000 Skopje, Republic of North Macedonia.}
\email{nikola.tuneski@mf.edu.mk}



\begin{abstract}
In this paper we give improved, probably not sharp, upper bounds of the Hankel determinant of third order for various classes of univalent functions.
\end{abstract}


\maketitle

\section{Introduction and preliminaries}

\smallskip

A complex-valued function of a complex variable is called univalent in a certain domain if it does not take the same value twice on that domain. The theory of univalent functions is over a century old and is rich of different types of results. There are two major approaches for their study. One is to investigate the properties of the subclasses of univalent functions and the other is to deliver criteria for a function to belong to those subclasses.

In recent period, a problem that follows the first approach, that is the problem of finding upper bound, preferably sharp, of the Hankel determinant for classes of univalent functions, is being rediscovered and attracts significant attention among the mathematicians working in the field.

For a function $f$ from the class ${\mathcal A}$ of analytic functions in the unit disk $\ID := \{ z\in \IC:\, |z| < 1 \}$   normalized such that $f(0)=f'(0)-1=0$, i.e.,
$f(z)=z+a_2z^2+a_3z^3+\cdots$, the $qth$ Hankel determinant is defined for $q\geq 1$, and
$n\geq 1$ by
\[
        H_{q}(n) = \left |
        \begin{array}{cccc}
        a_{n} & a_{n+1}& \ldots& a_{n+q-1}\\
        a_{n+1}&a_{n+2}& \ldots& a_{n+q}\\
        \vdots&\vdots&~&\vdots \\
        a_{n+q-1}& a_{n+q}&\ldots&a_{n+2q-2}\\
        \end{array}
        \right |.
\]
Thus, the second Hankel determinant is $H_{2}(2)= a_2a_4-a_{3}^2$ and the third is
\[ H_3(1) =  \left |
        \begin{array}{ccc}
        1 & a_2& a_3\\
        a_2 & a_3& a_4\\
        a_3 & a_4& a_5\\
        \end{array}
        \right | = a_3(a_2a_4-a_{3}^2)-a_4(a_4-a_2a_3)+a_5(a_3-a_2^2).
\]
The concept of Hankel determinant finds its application in the theory of singularities (see \cite{dienes}) and in the study of power series with integral coefficients.

Finding sharp estimate of the Hankel determinant for the whole class of univalent functions is very difficult task and the most significant result in that direction
is the one of Hayman (\cite{hayman-68}) who showed that $|H_2(n)|\le An^{1/2}$, where A is an absolute constant, and that this rate of growth is the best possible.
On the other hand, the sharp upper bound of the second Hankel determinant for subclasses of univalent functions is more accessible.
For example, for the classes of starlike and convex functions it turns out to be 1 and $1/8$, respectively (Janteng et al., \cite{janteng-07}).
The classes of starlike and convex functions are defined respectively with
\[ \mathcal{S}^\ast = \left\{ f\in\A:\frac{zf'(z)}{f(z)}\prec \frac{1+z}{1-z}\right\}  \]
and
\[ \mathcal{C} = \left\{ f\in\A:1+\frac{zf''(z)}{f'(z)}\prec \frac{1+z}{1-z}\right\},  \]
where "$\prec$" denotes the usual subordination defined by: $f\prec g$, if, and only if, $f$ and $g$ are analytic in $\D$ and there exists a Schwarz function $\omega$
(analytic in $\D$, $\omega(0) = 0$ and $|\omega(z)| < 1$), such that $f(z) = g(\omega(z))$ for $z\in\D$. If $g$ is univalent, then $f\prec g$ is equivalent to $f(0)=g(0)$ and $f(\D)\subseteq g(\D)$.

The case of the third Hankel determinant was first studied by Babaloa in \cite{babaloa} and it appears to be much more complicated and the existing results are usually not sharp.
Some of the few known sharp estimates is $1/9$ for the classes $\mathcal{S}^\ast(1/2)$ (the class of starlike functions of order 1/2, such that $f\in\es^\ast(1/2)\subset \A$ if, and only if,
$\operatorname{Re}[zf'(z)/f(z)]>1/2$, $z\in\D$) given by Zaprawa et al. in \cite{zaprawa} and $4/135$  for the class $\mathcal{C}$ given by Kowalczyk et al. in \cite{Kowalczyk-18}.
For the class $\U$ defined by
\[\U =  \left\{ f\in\A: \left| \left [\frac{z}{f(z)} \right]^{2}f'(z)-1 \right|<1,\, z\in\D \right\}, \]
the authors have proven (\cite{MONT-2018-2}) sharp bounds to be
$$|H_{2}(2)|\leq1\quad \mbox{and}\quad  |H_{3}(1)|\leq \frac{1}{4}.$$

Other results on the second and the third Hankel determinant can be found in \cite{opoola,kwon-1,kwon-2,lee,MONT-2019-1,selvaraj,sokol,DTV-book,krishna}.

In this paper we will study the class of starlike functions, as well as, the following classes
\[\es^\ast_s= \left\{ f\in\A: \frac{2zf'(z)}{f(z)-f(-z)} \prec \frac{1+z}{1-z} \right\}, \]
\[\es^\ast_e= \left\{ f\in\A: \frac{zf'(z)}{f(z)}\prec  e^z\right\} \]
and
\[\es_q^\ast= \left\{ f\in\A: \frac{zf'(z)}{f(z)}\prec  z+\sqrt{1+z^{2}} \right\}. \]

The class $\es^\ast_s$ consists of functions starlike with respect to symmetric points introduced by Sakaguchi in \cite{saka} where he proved that these functions are close-to-convex,
and hence univalent. Mishra et al. in \cite{mishra} proved sharp bound of $|H_2(2)|$ on this class to be 1,
and non-sharp bound of $H_3(1)$ to be $5/2$.

For the class $\es^\ast_e$, Shi et al. in \cite{shi}  proved that $|H_3(1)|\le 0.50047781$.

In their paper \cite{opoola} Bello and Opoola considered the class $\mathcal{S}_q$ and found that $|H_{2}(2)|\leq\frac{39}{48}$ which was improved by the authors (\cite{MONT-2018-1}) to its sharp value $1/4$.
No estimate for $|H_3(1)|$ is known for this class. It can be verified that the function $z+\sqrt{1+z^{2}}$ maps the unit disk into the right half plane, so the class $\es_q^\ast$ contains only starlike functions,
i.e., $\es_q^\ast\subset\es^\ast$.

For this three classes we improve the existing estimates  of the modulus of the third Hankel determinant and we do it by using different approach than the common one.

Namely, the current research on the estimates of the Hankel determinant is done by application of a result about coefficients of Carath\'{e}odory functions
(functions from with positive real part on the unit disk) that involves Toeplitz determinants. This result is due to Carath\'{e}odory and Toeplitz, a proof of which can be found in Grenander and Szeg\H{o} (\cite{granader}).
The result itself can be found in \cite[Theorem 3.1.4, p.26]{DTV-book}.

On the other hand, in this paper, as we already did in \cite{MONT-2019-1}, we use different method, based on the estimates of the coefficients of  Shwartz function
due to Prokhorov and Szynal (\cite{Prokhorov-1984}). Here is that result.

\begin{lem}\label{lem-prok}
Let $\omega(z)=c_{1}z+c_{2}z^{2}+\cdots $ be a Schwarz function. Then, for any real numbers $\mu$ and $\nu$
the following sharp estimate holds
 $$\Psi(\omega)=|c_{3}+\mu c_{1}c_{2}+\nu c_{1}^{3}|\leq \Phi(\mu,\nu),$$
where $\Phi(\mu,\nu)$ is given in complete form in \cite[Lemma 2]{Prokhorov-1984}, and here we will use only
$$\Phi(\mu,\nu) =  \left\{\begin{array}{cc}
1, & (\mu,\nu) \in D_1\cup D_2\cup \{(2,1)\}  \\[2mm]
|\nu|, & (\mu,\nu) \in \cup_{k=3}^7 D_k
\end{array}
\right.,$$
where
\smallskip

$ D_1 = \left\{(\mu,\nu):|\mu|\le\frac12,\, -1\le\nu\le1 \right\}, $
\smallskip

$ D_2 = \left\{(\mu,\nu):\frac12\le|\mu|\le2,\, \frac{4}{27}(|\mu|+1)^3-(|\mu|+1)\le\nu\le1 \right\}, $
\smallskip

$ D_3 = \left\{(\mu,\nu):|\mu|\le\frac12,\, \nu\le-1 \right\}, $
\smallskip

$ D_4 = \left\{(\mu,\nu):|\mu|\ge\frac12,\, \nu\le-\frac23(|\mu|+1) \right\}, $
\smallskip

$ D_5 = \left\{(\mu,\nu):|\mu|\le2,\, \nu\ge1 \right\}, $
\smallskip

$ D_6 = \left\{(\mu,\nu):2\le|\mu|\le4,\, \nu\ge\frac{1}{12}(\mu^2+8) \right\}, $
\smallskip

$ D_7 = \left\{(\mu,\nu):|\mu|\ge4,\, \nu\ge\frac23(|\mu|-1) \right\}. $
\end{lem}

We will also use the following,  almost forgotten result of Carleson (\cite{carlson}).

\begin{lem}\label{lem-carl}
Let $\omega(z)=c_{1}z+c_{2}z^{2}+\cdots $ be a Schwarz function. Then
\[|c_2|\le1-|c_1|^2,\quad  |c_3|\le 1-|c_1|^2-\frac{|c_2|^2}{1+|c_1|} \quad\mbox{and}\quad |c_4|\le1-|c_1|^2 -|c_2|^2. \]
\end{lem}

\medskip

\section{Main results}

\smallskip

In this section we give improved, probably non-sharp estimates of the third Hankel determinant for classes $\es^\ast$, $\es^\ast_s$, $\es^\ast_e$ and $\es_q^\ast$.

\begin{thm}\label{main-thm}
Let $f\in\A$ is of the form $f(z)=z+a_2z^2+a_3z^3+\cdots$.
\begin{itemize}
  \item[($i$)] If $f\in\es^\ast$, then $|H_3(1)|\le  0.777987\ldots$.
  \item[($ii$)] If  $f\in\es^\ast_s$, then $|H_3(1)|\le \frac14+  \frac{1}{3\sqrt3}=0.44245\ldots$.
  \item[($iii$)] If $f\in\es^\ast_e$, then $|H_3(1)|\le \frac{17}{72}=0.23611\ldots$.
  \item[($iv$)] If $f\in\es_q^\ast$, then $|H_3(1)|\le \frac{17}{72}=0.23611\ldots$.
\end{itemize}
\end{thm}

\begin{proof}$ $
\noindent ($i$) Let $f\in\es^\ast$. From the definitions of starlikeness and subordination we realize that there exists a function $\omega$ analytic in the unit disk such that $\omega(0)=0$, $|\omega(z)|<1$ for all $z\in\D$ and
  \[ \frac{zf'(z)}{f(z)} = \frac{1+\omega(z)}{1-\omega(z)}, \]
  leading to
  \[ zf'(z)(1-\omega(z)) = f(z)(1+\omega(z)).\]
  Using  $f(z)=z+a_2z^2+a_3z^3+\cdots$ and $\omega(z)=c_1z+c_2z^2+c_3z^3+\cdots$, by equating coefficients, we receive
  \[
  \begin{split}
  a_2 &= 2c_1, \\
  a_3 &= c_2+3c_1^2, \\
  a_4 &= \frac23\left(c_3+5c_1c_2+6c_1^3\right), \\
  a_5 &= \frac12\left(c_4+\frac{14}{3}c_1c_3+\frac{43}{3}c_1^2c_2+2c_2^2+10c_1^4\right),
  \end{split}
  \]
  and further
  \[
  \begin{split}
  H_3(1) &= \frac{1}{18}\left[ -8\left( c_3-\frac54c_1c_2 \right) -\frac{63}{8}c_1^2c_2^2 +6c_1^3\left(c_3+\frac12c_1c_2\right) +9(c_2-c_1^2)c_4 \right].
  \end{split}
  \]
  The last implies
  \[
  \begin{split}
  |H_3(1)| &\le \frac{1}{18}\left[ 8\left| c_3-\frac54c_1c_2 \right| +\frac{63}{8}|c_1|^2|c_2|^2 +6|c_1|^3\left|c_3+\frac12c_1c_2\right| +9|c_2-c_1^2||c_4| \right].
  \end{split}
  \]
  Now, from Lemma \ref{lem-prok} by using $\mu=-\frac{5}{4}$, $\nu =0$, $(\mu,\nu)\in D_2$; and $\mu=\frac12$, $\nu =0$, $(\mu,\nu)\in D_1$; we receive respectively
  \[ \left|c_3-\frac{5}{4}c_1c_2\right|\le 1 \quad\mbox{and}\quad  \left|c_3+\frac12 c_1c_2\right|\le1. \]
  Therefore, using additionally the estimate for $|c_4|$ from  Lemma \ref{lem-carl} we receive
  \[
  \begin{split}
  |H_3(1)|
  &\le \frac{1}{18}\left[ 8 +\frac{63}{8}|c_1|^2|c_2|^2 +6|c_1|^3 +9|c_2-c_1^2|(1-|c_1|^2-|c_2|^2) \right]\\
  &= \frac{1}{18}\left[ 8 -\frac{9}{8}|c_1|^2|c_2|^2 - 9|c_1|^2|c_2| +9|c_1|^2 +6|c_1|^3 - 9|c_1|^4 + 9|c_2| - 9|c_2|^3 \right]\\
  &= \frac{1}{18}\left[ 8 + h(|c_1|,|c_2|)\right],
      \end{split}
   \]
  where
  \[ h(x,y) = -\frac{9}{8}x^2y^2 - 9x^2y +9x^2 +6x^3 - 9x^4 + 9y - 9y^3 \]
  and $(x,y)\in \Omega=\{(x,y):0\le x\le1, 0\le y\le 1-x^2\}$.

  \medskip
  We continue with looking for the maximal value of $h$ on $\Omega$.

  \medskip

  The equation  $h'_y(x,y)=-\frac{9}{4}  \left[x^2 (y+4)+12 y^2-4\right]=0$ has positive solution $x=2\sqrt{\frac{1-3y^2}{4+y}}=:g(y)$ only when $0\le y\le \frac{\sqrt3}{3}=0.57735\ldots$. After replacing it in $h'_x$
  we receive
  \[ h'_x(g(y),y) = \frac{9 \sqrt{1-3 y^2} }{2 (y+4)^{3/2}} \left[16 \sqrt{(y+4)(1-3 y^2)}- y^3+180y^2-24y - 32\right] . \]
The equation $h'_x(g(y),y)=0$ has only two positive solutions $y_1=0.1541\ldots$ with $x_1=g(y_1)=0.94567\ldots$ and $y_1>1-x_1^2=0.1057\ldots$; and $y_2=\frac{1}{\sqrt3}=0.57735\ldots$ with $x_2=g(y_2)=0$, but none of these points $(x_1,y_1)$ and $(x_2,y_2)$ lies in the interior of $\Omega$. Thus, the function $h$ attains its maximum on the boundary of $\Omega$.

\medskip

For $x=0$ we have $h(0,y)=9 y(1-y^2)\le 2\sqrt3=3.4641\ldots$  for $y=\frac{1}{\sqrt3}$.

For $x=1$ we have $h(1,y)=-9 y^3-\frac{9 y^2}{8}+6\le 6$  for $y=0$.

For $y=0$ we have $h(x,0)=x^2(-9 x^2+6 x+9) \le 6$  for $x=1$.

For $y=1-x^2$ we have $h(x,1-x^2)= \frac{3}{8} x^2 \left(21 x^4-66 x^2+16x+45\right)=:r(x)$, with $r'(x)=-\frac94x(x-1)\left( 21x^3+21x^2-23x-15 \right)=0$ having unique solution on $(0,1)$, $x_*=0.948542\ldots$. So, $h(x,1-x^2)\le 6.00376\ldots$.

\medskip

For all the above analysis we can conclude that $h(x,y)\le 6.00376\ldots$ for all $(x,y)\in\Omega$, i.e., $|H_3(1)|\le  \frac{1}{18}(8+6.00376\ldots)=0.777987\ldots$.

\medskip

  \noindent ($ii$) Similarly as in ($i$) for a function $f$ from $\es^\ast_s$ we have that there exists  a Schwarz function $\omega$ such that
  \[ \frac{2zf'(z)}{f(z)-f(-z)} = \frac{1+\omega(z)}{1-\omega(z)}, \]
  i.e.,
  \[ 2zf'(z)(1-\omega(z)) = (f(z)-f(-z))(1+\omega(z)). \]
  For $f(z)=z+a_2z^2+a_3z^3+\cdots$ and $\omega(z)=c_1z+c_2z^2+c_3z^3+\cdots$, by equating coefficients, we receive
  \[
  \begin{split}
  a_2 &= c_1, \\
  a_3 &= c_2+c_1^2, \\
  a_4 &= \frac12\left(c_3+3c_1c_2+2c_1^3\right), \\
  a_5 &= \frac12\left(c_4+2c_1c_3+5c_1^2c_2+2c_2^2+2c_1^4\right),
  \end{split}
  \]
  and
  \[
  \begin{split}
  H_3(1) &= \frac14 \left(- c_3^2 + 2 c_1 c_2 c_3  +c_1^2 c_2^2   + 2 c_2 c_4\right) \\
         &= \frac14 \left[- (c_3 - c_1 c_2)^2  +2c_1^2 c_2^2 + 2c_2c_4\right].
  \end{split}
  \]
  Thus,
  \[ |H_3(1)| \le  \frac14 \left(|c_3 - c_1 c_2|^2  +2|c_1|^2 |c_2|^2 + 2|c_2||c_4|\right),\]
and by applying Lemma \ref{lem-prok} ($\mu=-1$ and $\nu=0$) and Lemma \ref{lem-carl}:
  \[
  \begin{split}
  |H_3(1)| &\le  \frac14 \left[1  +2|c_1|^2 |c_2|^2 + 2|c_2|(1-|c_1|^2-|c_2|^2)\right]\\
  &=  \frac14 \left[1  -2|c_1|^2 |c_2|(1-|c_2|)   + 2|c_2| -2|c_2|^3 \right]\\
  &\le  \frac14 \left[1 + 2|c_2| -2|c_2|^3 \right]\\
  &\le \frac14+\frac{1}{3\sqrt3}.
  \end{split}
  \]

\smallskip

  \noindent ($iii$) If $f\in\es^\ast_e$, then for a Schwarz function $\omega$, $\frac{zf'(z)}{f(z)}= e^{\omega(z)}$  and $ zf'(z)= e^{\omega(z)} f(z)$.
Thus, for   $f(z)=z+a_2z^2+a_3z^3+\cdots$ and $\omega(z)=c_1z+c_2z^2+c_3z^3+\cdots$, by equating coefficients, we receive
  \[
  \begin{split}
  a_2 &= c_1, \\
  a_3 &= \frac12c_2+\frac34c_1^2, \\
  a_4 &= \frac13\left(c_3+\frac52c_1c_2+\frac{17}{12}c_1^3\right), \\
  a_5 &= \frac14\left(c_4+\frac73c_1c_3+\frac{10}{3}c_1^2c_2+\frac{19}{18}c_2^2+c_1^4\right),
    \end{split}
  \]
and
  \[
  \begin{split}
   H_3(1) &= \frac{5}{72}c_1c_2c_3-\frac{5}{72}c_1^2c_2^2+\frac{13}{864}c_1^4c_2+\frac{17}{432}c_1^3c_3-\frac19c_3^2-\frac{13}{5184}c_1^6 \\
   &\quad +\frac{1}{16}(2c_2-c_1^2)c_4 \\
   &= -\frac19\left( c_3-\frac{15}{16}c_1c_2 \right)^2 -\frac{15}{256}c_1^2c_2^2  + \frac{17}{432}c_1^3\left(c_3+\frac{13}{34}c_1c_2-\frac{13}{204}c_1^3\right) \\
   &\quad + \frac{1}{16}(2c_2-c_1^2)c_4.
  \end{split}
  \]
  From here,
  \[
  \begin{split}
  |H_3(1)| &\le  \frac19\left| c_3-\frac{15}{16}c_1c_2 \right|^2 +\frac{15}{256}|c_1|^2|c_2|^2  + \frac{17}{432}|c_1|^3\left|c_3+\frac{13}{34}c_1c_2-\frac{13}{204}c_1^3\right| \\
   &\quad + \frac{1}{16}(2|c_2|+|c_1|^2)|c_4|.
   \end{split}
   \]
  Choosing $(\mu,\nu)=(-5/16,0)\in D_1$  and $(\mu,\nu)=\left(\frac{13}{34},-\frac{13}{204}\right)\in D_1$ in Lemma \ref{lem-prok} we receive
  \[ \left| c_3-\frac{15}{16}c_1c_2 \right|\le1 \quad\mbox{and}\quad  \left|c_3+\frac{13}{34}c_1c_2-\frac{13}{204}c_1^3\right| \le 1,\]
  respectively. This, together with the estimate for $|c_4|$ from Lemma \ref{lem-carl} gives
    \[
  \begin{split}
  |H_3(1)|
  &\le  \frac19+\frac{15}{256}|c_1|^2|c_2|^2 + \frac{17}{432}|c_1|^3+ \frac{1}{16}(2|c_2|+|c_1|^2)(1-|c_1|^2-|c_2|^2) \\
  &\le  \frac19-\frac{1}{256}|c_1|^2|c_2|^2 + \frac{17}{432}|c_1|^3+ \frac{1}{8}|c_2| - \frac18|c_1|^2|c_2|\\
  &\quad -\frac18|c_2|^3+\frac{1}{16}|c_1|^2 - \frac{1}{16}|c_1|^4\\
  &\le  \frac19-\frac{1}{256}|c_1|^2|c_2|^2 + \frac{17}{432}|c_1|^3+ \frac{1}{8}(1-|c_1|^2) - \frac18|c_1|^2|c_2|\\
  &\quad -\frac18|c_2|^3+\frac{1}{16}|c_1|^2 - \frac{1}{16}|c_1|^4\\
  &\le  \frac{17}{72}-\frac{1}{256}|c_1|^2|c_2|^2 - \frac{1}{16}|c_1|^2 \left( 1-\frac{17}{27}|c_1|+|c_1|^2 \right) - \frac{1}{8}|c_1|^2|c_2| - \frac18|c_2|^3\\
  &\le  \frac{17}{72},
  \end{split}
   \]
  since all other terms are negative.

  \smallskip

  \noindent ($iv$) Again, if $f\in\es_q^\ast$, then $\frac{zf'(z)}{f(z)}=  \omega(z)+\sqrt{1+\omega^2(z)} $ for some Schwarz function $\omega$, and for $f(z)=z+a_2z^2+a_3z^3+\cdots$
  and $\omega(z)=c_1z+c_2z^2+c_3z^3+\cdots$, by equating the coefficients, we receive
    \[
  \begin{split}
  a_2 &= c_1, \\
  a_3 &= \frac12c_2+\frac34c_1^2, \\
  a_4 &= \frac13\left(c_3+\frac52c_1c_2+\frac54 c_1^3\right), \\
  a_5 &= \frac14\left(c_4+\frac73c_1c_3+\frac{17}{6}c_1^2c_2+c_2^2+\frac23c_1^4\right),
    \end{split}
  \]
i.e.,
  \[
  \begin{split}
   H_3(1) &= -\frac19 \left(c_3-\frac{5}{16}c_1c_2\right)^2 - \frac{31}{256}c_1^2c_2^2 + \frac{11}{144}     c_1^3\left(c_3+\frac{5}{11}c_1c_2-\frac{7}{44}c_1^3\right) \\
   &\quad + \frac{1}{8}\left(c_2-\frac12c_1^2\right)c_4.
  \end{split}
  \]
So,
  \[
  \begin{split}
   |H_3(1)|
   &\le \frac19 \left|c_3-\frac{5}{16}c_1c_2\right|^2 + \frac{31}{256}|c_1|^2|c_2|^2 + \frac{11}{144}     |c_1|^3\left|c_3+\frac{5}{11}c_1c_2-\frac{7}{44}c_1^3\right| \\
   &\quad + \frac{1}{8}\left(|c_2|+\frac12|c_1|^2\right)|c_4|
  \end{split}
  \]
Now, Lemma \ref{lem-prok} for $(\mu,\nu)=\left( -\frac{5}{16},0 \right)\in D_1$  and $(\mu,\nu)=\left( \frac{5}{11}, -\frac{7}{44} \right)\in D_1$  brings
\[\left|c_3-\frac{5}{16}c_1c_2\right|\le1 \quad\mbox{and}\quad \left|c_3+\frac{5}{11}c_1c_2-\frac{7}{44}c_1^3\right|\le1,\]
which together with the estimate for $|c_4|$ from Lemma \ref{lem-carl} implies
  \[
  \begin{split}
   |H_3(1)|
   &\le \frac19  + \frac{31}{256}|c_1|^2|c_2|^2 + \frac{11}{144}  |c_1|^3 + \frac{1}{8}\left(|c_2|+\frac12|c_1|^2\right)(1-|c_1|^2-|c_2|^2) \\
   &= \frac19  + \frac{15}{256}|c_1|^2|c_2|^2 + \frac{11}{144}  |c_1|^3 +  \frac{1}{8}|c_2| - \frac18|c_1|^2|c_2| - \frac18|c_2|^3 \\
   &\quad +\frac{1}{16}|c_1|^2 - \frac{1}{16}|c_1|^4\\
   &\le \frac19  + \frac{15}{256}|c_1|^2|c_2|^2 + \frac{11}{144}  |c_1|^3 +  \frac{1}{8}(1-|c_1|^2) - \frac18|c_1|^2|c_2| - \frac18|c_2|^3 \\
   &\quad +\frac{1}{16}|c_1|^2 - \frac{1}{16}|c_1|^4\\
   &\le \frac{17}{72} -  \frac18|c_1|^2|c_2|\left(1-\frac{15}{32}|c_2| \right) - \frac{1}{16}  |c_1|^2 \left( |c_1|^2-\frac{11}{9}|c_1|+1 \right) -\frac18|c_2|^3\\
   &\le  \frac{17}{72}.
  \end{split}
  \]
\end{proof}

\smallskip

Estimates of the third Hankel determinant for classes $\es^\ast$ and $\es_q^\ast$ (($i$) and ($iv$) from the previous theorem) were not given before, while
estimates for classes $\es_s^\ast$ and $\es_e^\ast$ ($0.26897\ldots$ and $0.125$ from ($ii$) and ($iii$), respectively) are better then the currently known
obtained by Carath\'{e}odory functions ($5/2$ from \cite{mishra} and $0.50047781\ldots$ from \cite{shi}, respectively).


\smallskip


\medskip


\begin{thebibliography}{99}


\bibitem{babaloa}
Babalola K.O., On $H_3(1)$ Hankel determinant for some classes of univalent functions, \textit{Inequal. Theory Appl.} 2010, 6, 1-–7.

\bibitem{opoola}
Bello R.A., Opoola T.O., Upper Bounds for Fekete-Szego functions and the Second Hankel
Determinant for a Class of Starlike functions, \emph{IOSR Journal of Mathematics}, \textbf{13} (2) Ver. V, (2017) 34--39.

\bibitem{carlson}
Carlson F., Sur les coefficients d'une fonction born\'{e}e dans le cercle unit\'{e}, \textit{Ark. Mat. Astr. Fys.}, 27A, No. 1, 8 pp, (1940).

\bibitem{dienes}
Dienes P., The Taylor Series: An Introduction to the Theory of Functions of a Complex Variable; NewYork-Dover: Mineola, NY, USA, 1957.


\bibitem{granader}
Grenander U., Szeg\H{o} G., \textit{Toeplitz forms and their applications}, California Monographs in Mathematical Sciences. University of California Press, Berkeley-Los Angeles, 1958.

\bibitem{hayman-68}
Hayman W.K., On the second Hankel determinant of mean univalent functions, \textit{Proc. London Math. Soc.} (3), 18:77–94, 1968.

\bibitem{janteng-07}
Janteng A.,  Halim S.A., Darus M., Hankel determinant for starlike and convex functions, \textit{Int. J. Math. Anal. (Ruse)}, 1(13-16):619–625, 2007.

\bibitem{Kowalczyk-18}
Kowalczyk B.,  Lecko A., Sim Y.J., The sharp bound of the Hankel determinant of the third kind for convex functions. \textit{Bull. Aust. Math. Soc.},  97 (2018), no. 3, 435--445.

\bibitem{kwon-1}
Kwon O.S.,  Lecko A.,  Sim Y.J., The bound of the Hankel determinant of the third kind for starlike functions, \textit{Bull. Malays. Math. Sci. Soc.} 42 (2019), no. 2, 767--780.

\bibitem{kwon-2}
Kwon O.S.,  Sim Y.J.,  The Sharp Bound of the Hankel Determinant of the Third Kind for Starlike Functions with Real Coefficients, \textit{Mathematics} 2019, 7, 721.

\bibitem{lee}
Lee S.E., Ravichandran V., Supramaniam S., Bounds for the second Hankel determinant of certain univalent functions, \textit{J. Inequal. Appl.} 2013, 2013:281, 17 pp.

\bibitem{mishra}
Mishra A.K.,  Prajapat J.K., Maharana S., Bounds on Hankel determinant for starlike and convex functions with respect to symmetric points, \textit{Cogent Math.} 3 (2016), Art. ID 1160557, 9 pp.

\bibitem{MONT-2019-1}
Obradovi\'{c} M., Tuneski N., Hankel determinant for a class of analytic functions, \textit{Advances in Mathematics: Scientific Journal}, Vol. 8 No. 1 (2019), 1--6.

\bibitem{MONT-2018-1}
Obradovi\'{c} M., Tuneski N., Hankel determinant of second order for some  classes of analytic  functions, \textit{preprint}.

\bibitem{MONT-2018-2}
Obradovi\'{c} M., Tuneski N., Some properties of the class $\mathcal{U}$, \textit{Annales. Universitatis Mariae Curie-Skłodowska. Sectio A - Mathematica}, accepted.

\bibitem{Prokhorov-1984}
Prokhorov D.V., Szynal J., Inverse coefficients for $(\alpha ,\beta )$-convex functions, \emph{Ann. Univ. Mariae Curie-Sk{\l}odowska Sect. A}, 35 (1981), 125--143 (1984).

\bibitem{saka}
Sakaguchi K.,  On a certain univalent mapping. \textit{Journal of the Mathematical Society of Japan}, (1959) 11, 72-–75.

\bibitem{selvaraj}
Selvaraj C., Kumar T.R.K., Second Hankel determinant for certain classes of analytic functions, \textit{Int. J. Appl. Math.} 28 (2015), no. 1, 37--50.

\bibitem{shi}
Shi L., Srivastava H.M., Arif M., Hussain S.,  Khan H., An Investigation of the Third Hankel Determinant Problem for Certain Subfamilies of Univalent Functions Involving the Exponential Function, \textit{Symmetry} 2019, 11, 598.

\bibitem{sokol}
Sokol J., Thomas D.K., The second Hankel determinant for alpha-convex functions, Lith.Math. J., \textbf{58} (2), (2018),  212–-218.

\bibitem{DTV-book}
Thomas D.K., Tuneski N., Vasudevarao A., Univalent Functions: A Primer, \emph{De Gruyter Studies in Mathematics} {\bf 69}, De Gruyter, Berlin, Boston, 2018.

\bibitem{krishna}
Vamshee Krishna D., Venkateswarlu B., Ramreddy T., Third Hankel determinant for starlike and convex functions with respect to symmetric points, \textit{Ann. Univ. Mariae Curie-Skłodowska Sect. A} 70 (2016), no. 1, 37--45.

\bibitem{zaprawa}
Zaprawa P., Third Hankel determinants for subclasses of univalent functions, \textit{Mediterr. J. Math.} 14 (2017), no. 1, Art. 19, 10 pp.
\end{thebibliography}
\end{document}